\documentclass[11pt,a4paper]{article}

\usepackage{theorem,enumerate}
\usepackage{amsmath,latexsym,amssymb,amsfonts}
\usepackage{eucal}
\usepackage{mathrsfs}
\usepackage{array}
 \usepackage{graphicx, color}

\newcounter{Scounter}
\theorembodyfont{\normalfont\slshape}
\newtheorem{thm}{Theorem}

\newtheorem{lemma}[thm]{Lemma}
\newtheorem{rem}[thm]{Remark}

\newtheorem{claim}{Claim}

\newcommand{\Case}[1]{\noin{\bf Case #1}}

\newcommand{\qed}{{$\quad\square$\vs{3.6}}}

\newcommand{\vs}[1]{\vspace*{#1 mm}}

\newcommand{\noin}{\noindent}

\def\M{{ \mathcal{M}}}

\bfseries\normalfont

\makeatletter
\def\thanks#1{%
   \footnotemark
   \edef\@tempa{\noexpand\noexpand\noexpand\footnotetext[\the\c@footnote]}%
   \toks@\expandafter{\@thanks}%
   \toks\tw@{{#1}}
   \xdef\@thanks{\the\toks@\@tempa\the\toks\tw@}}
\makeatother

\begin{document}

\title{A note on a Brooks' type theorem for DP-coloring}

\author{
Seog-Jin Kim\thanks{Department of Mathematics Education, Konkuk University,
Korea.
e-mail: {\tt skim12@konkuk.ac.kr
}}\thanks{
This research was supported by Basic Science Research Program through the National Research Foundation of Korea(NRF) funded by the Ministry of Education(NRF-2015R1D1A1A01057008).}
and 
Kenta Ozeki\thanks{Faculty of Environment and Information Sciences,
Yokohama National University,
Japan. 
e-mail: {\tt ozeki-kenta-xr@ynu.ac.jp}
}
}

\date{July 13, 2017}
\maketitle

\begin{abstract}
Dvo\v{r}\'{a}k and Postle \cite{DP} 
introduced a \textit{DP-coloring} of a simple graph
as a generalization of a list-coloring.
They proved a Brooks' type theorem for a DP-coloring,
and 
Bernshteyn, Kostochka and Pron \cite{BKP}
extended it to a DP-coloring of multigraphs.
However,
detailed structure when a multigraph does not admit a DP-coloring was not specified in \cite{BKP}.
In this note,
we make this point clear
and give the complete structure.
This is also motivated by 
the relation to signed coloring 
of signed graphs.
\end{abstract}


\section{Introduction}

\subsection{List-coloring and DP-coloring}

We denote by $[k]$ the set of integers from $1$ to $k$.
A $k$-coloring of a graph $G$
is a mapping $f : V(G) \rightarrow [k]$
such that $f(u) \not= f(v)$ for any $uv \in E(G)$.
The minimum integer $k$
such that $G$ admits a $k$-coloring 
is called the \textit{chromatic number of $G$},
and denoted by $\chi(G)$.

A \textit{list assignment} $L : V(G) \rightarrow 2^{[k]}$ of $G$
is a mapping that assigns a set of colors to each vertex.
A proper coloring $f: V(G) \rightarrow Y$ where $Y$ is a set of colors
is called an \textit{$L$-coloring} of $G$ if
$f(u) \in L(u)$ for any $u \in V(G)$.
A list assignment $L$ is called a \textit{$t$-list assignment} if 
$|L(u)| \geq t$ for any $u \in V(G)$.
The \textit{list-chromatic number} or the \textit{choice number} of $G$,
denoted by $\chi_{\ell}(G)$,
is the minimum integer $t$
such that $G$ admits an $L$-coloring for each $t$-list assignment $L$.


Since a $k$-coloring corresponds to an $L$-coloring
with $L(u) = [k]$ for any $u \in V(G)$,
we have $\chi(G) \leq \chi_{\ell}(G)$.
It is well-known that 
there are infinitely many graphs $G$ 
satisfying $\chi (G) < \chi_{\ell}(G)$,
and the gap can be arbitrary large:
Consider for example,
the complete bipartite graph $K_{t,t^t}$,
which satisfies 
$2 = \chi (K_{t,t^t}) < \chi_{\ell}(K_{t,t^t}) = t+1$.
A list assignment $L$ is called a \textit{degree-list assignment} if 
$|L(u)| \geq d_G(u)$ for any $u \in V(G)$,
where $d_G(u)$ denotes the degree of $u$ in $G$.
A graph $G$ is said to be \textit{degree-choosable}
if $G$ admits an $L$-coloring for degree-list assignment.
A Brooks' type theorem for degree-choosability was shown
by Borodin \cite{Borodin}, and independently Erd\H{o}s, Rubin, and Taylor \cite{ERT}.
See also \cite{KSW} for a shorter proof.

\begin{thm}
\label{BERT}
A connected graph $G$ is not degree-choosable
if and only if each block of $G$ is isomorphic to $K_n$ for some integer $n$
or $C_n$ for some odd integer $n$.
\end{thm}

Furthermore,
it is known that 
complete graphs and odd cycles
do not have an $L$-coloring
for a degree-list assignment $L$
only when 
all vertices have same list assignment
of size exactly their degree.

In order to consider some problems on list chromatic number,
Dvo\v{r}\'{a}k and Postle \cite{DP} considered 
a generalization of a list-coloring.
They call it a \textit{correspondence coloring},
but we call it a \textit{DP-coloring},
following Bernshteyn, Kostochka and Pron \cite{BKP}.
It was first proposed for a simple graph,
and then extended to a multigraph in \cite{BKP}.

Let $G$ be a multigraph (possibly having multiple edges but no loops)
and $L$ be a list assignment of $G$.
For each pair of vertices $u$ and $v$ in $G$,
let $M_{L,uv}$ be the union of $\mu_G(uv)$ matchings
between $\{u\} \times L(u)$ and $\{v\} \times L(v)$,
where $\mu_G(uv)$ is the multiplicity of $uv$ in $G$.
Note that if $u$ and $v$ are not connected by an edge in $G$,
then $\mu_G(uv)= 0$ and $M_{L,uv}$ is an empty set.
With abuse of notation,
we sometimes regard $M_{L,uv}$ 
as a bipartite graph
between $\{u\} \times L(u)$ and $\{v\} \times L(v)$
of maximum degree at most $\mu_G(uv)$.

Let $\M_{L} = \big\{M_{L,uv} : uv \in E(G) \big\}$,
which is called a \textit{matching assignment over $L$}.
Then a graph $H$ is said to be the \textit{$\M_{L}$-cover} of $G$
if it satisfies all the following conditions:
\begin{enumerate}[{\upshape (i)}]
\item
The vertex set of $H$ is 
$\bigcup_{u \in V(G)} \big(\{u\} \times L(u)\big) 
= \big\{(u,c): u \in V(G), \ c \in L(u)\big\}$.
\item
For any $u \in V(G)$,
the set $\{u\} \times L(u)$ is a clique in $H$.
\item
For any two vertices $u$ and $v$ in $G$,
$\{u\} \times L(u)$
and 
$\{v\} \times L(v)$
induce in $H$ the graph obtained from $M_{L,uv}$
by adding those edges defined in (ii).
\end{enumerate}

(See Figure \ref{example_fig} for an example.)
\bigskip

\begin{figure}
\centering
\input{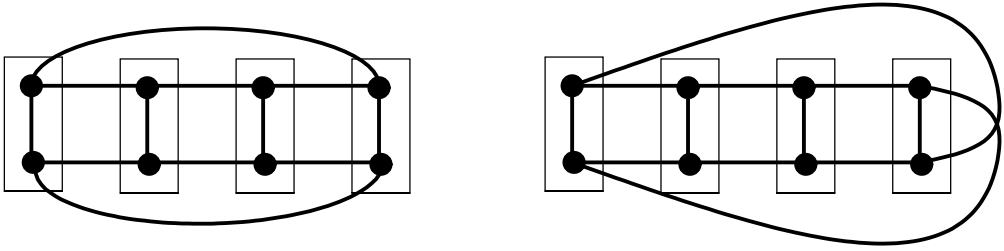}
\caption{Two examples of the $\M_L$-cover of $C_4$
such that $|L(u)| = 2$ for any vertex $u$.
Each thin rectangle represents $\{u\} \times L(u)$
for some vertex $u$.
In fact,
the cycle $C_4$ admits an $\M_L$-coloring for the left,
while does not for the right.}
\label{example_fig}
\end{figure}

An $\M_{L}$-coloring of $G$ is an independent set $I$ in 
the $\M_{L}$-cover with $ |I| = |V(G)|$.
The \textit{DP-chromatic number},
denoted by $\chi_{\text{DP}}(G)$,
is the minimum integer $t$
such that $G$ admits an $\M_{L}$-coloring
for each $t$-list assignment $L$ and each matching assignment $\M_{L}$ over $L$.
%

Note that 
when $G$ is a simple graph and 
$$M_{L,uv} = \big\{(u,c)(v,c): c \in L(u) \cap L(v)
\big\}
$$
for any edge $uv$ in $G$,
then $G$ admits an $L$-coloring
if and only if 
$G$ admits an $\M_{L}$-coloring.
Furthermore,
when $L(u) = [k]$ for each $u \in V(G)$
(that is, ~when we consider an ordinal $k$-coloring),
then the $\M_L$-cover of $G$ 
is isomorphic to the graph $G \square K_k$,
which is the \textit{Cartesian product} of $G$ and the complete graph $K_k$.
Recall that the Cartesian product $H_1 \square H_2$
of graphs $H_1$ and $H_2$
is the graph with $V(H_1 \square H_2) = V(H_1) \times V(H_2)$
and two vertices $(u_1,u_2)$ and $(v_1,v_2)$ are adjacent
if and only if 
either $u_1 = v_1$ and $u_2v_2 \in E(H_2)$
or $u_1 v_1 \in E(H_1)$ and $u_2 = v_2$.
In this case,
we see that
$G$ admits a $k$-coloring
if and only if $G \square K_k$ contains
an independent set of size $|V(G)|$.
According to \cite{BK},
this fact was pointed out by Plesnevi\v{c} and Vizing \cite{PV}.

The relation in the previous paragraph
implies $\chi_{\ell}(G) \leq \chi_{\text{DP}}(G)$.
There are infinitely many simple graphs $G$ 
satisfying $\chi_{\ell} (G) < \chi_{\text{DP}}(G)$:
As we will see in Theorems \ref{BERT} and \ref{mainthm},
$\chi(C_n) = \chi_{\ell}(C_n) = 2 < 3 = \chi_{\text{DP}}(C_n)$
for an even integer $n$.
Furthermore,
the gap $\chi_{\text{DP}}(G) - \chi_{\ell} (G)$
can be arbitrary large.
For example,
Bernshteyn \cite{Bernshteyn} showed that 
for a simple graph $G$ with average degree $d$,
we have $\chi_{\text{DP}}(G) = \Omega(d / \log d)$,
while Alon \cite{Alon} proved 
that $\chi_{\ell}(G) = \Omega(\log d)$ and 
the bound is sharp.
See \cite{BK2} for more detailed results.
Recently, there are some works on DP-colorings;
see \cite{Bernshteyn, BK, BKZ, DP}.

Bernshteyn, Kostochka and Pron \cite{BKP} proved 
a Brooks' type theorem for DP-coloring of multigraphs.
For a multigraph $G$ and an integer $t$,
we denote by $G^t$ the multigraph obtained from $G$ 
by replacing each edge with a set of $t$ multiples edges.
A multigraph $G$ is said to be \textit{degree-DP-colorable}
if $G$ admits an $\M_{L}$-coloring
for each degree-list assignment $L$ and 
each matching assignment $\M_{L}$ over $L$.

The following theorem gives
a Brooks' type theorem for a DP-coloring.
This is an extension of Theorem \ref{BERT}.

\begin{thm}[Bernshteyn, Kostochka and Pron \cite{BKP}]
\label{BKPthm}
A connected multigraph $G$ is not degree-DP-colorable
if and only 
each block of $G$ is $K_n^t$ or $C_n^t$
for some $n$ and $t$.
\end{thm}

However,
Theorem \ref{BKPthm}
does not explain the $M_L$-colorability 
for a matching assignment $\M_L$
when every block of $G$ is $K_n^t$ or $C_n^t$
for some $n$ and $t$.
The purpose of this paper is to make this point clear
and give the complete structure.
As explained in the next subsection,
it is important work because of the relation to
some results on signed colorings of signed graphs.

\subsection{Singed colorings of signed graphs}

Here, we give a relationship to signed colorings of signed graphs.
A \textit{signed graph} $(G,\sigma)$
is a pair of a multigraph $G$ 
and a mapping $\sigma: E(G) \rightarrow \{1, -1\}$,
which is called a \textit{sign}.
For an integer $k$,
let 
$$N_k = 
\begin{cases}
\{0, \pm 1, \dots , \pm r\} & \text{if $k$ is an odd integer with $k = 2r +1$,} \\
\{\pm 1, \dots , \pm r\} & \text{if $k$ is an even integer with $k = 2r$.}
\end{cases}
$$
A \textit{signed $k$-coloring} of a signed graph $(G,\sigma)$
is a mapping $f : V(G) \rightarrow N_k$
such that $f(u) \not= \sigma(uv) f(v)$
for each $uv \in E(G)$.
The minimum integer $k$
such that a signed graph $(G,\sigma)$ admits a signed $k$-coloring
is called the \textit{signed chromatic number of $(G,\sigma)$}.
This was first defined by Zaslavsky \cite{Zaslavsky}
with slightly different form,
and then modified by M\'{a}\v{c}ajov\'{a}, Raspaud, and \v{S}koviera \cite{MRS}
to the above form
so that it would be a natural extension of an ordinally coloring.

We here point out that
a signed coloring of a signed graph $(G,\sigma)$
is a special case of a DP-coloring of $G$.
Let $L$ be the list assignment of $G$
with $L(u) = N_k$ for any vertex $u$ in $G$.
Then for an edge $uv$ in $G$,
let 
$$M_{L,uv} = 
\begin{cases}
\big\{(u,i)(v,i) : i \in N_k\big\} & \text{ if $\sigma(uv) = 1$,}\\
\big\{(u,i)(v,-i) : i \in N_k\big\} & \text{ if $\sigma(uv) = -1$.}
\end{cases}
$$
With this definition,
it is easy to see that
the signed graph $(G,\sigma)$ admits a signed $k$-coloring
if and only if 
the multigraph $G$ admits an $\M_{L}$-coloring.

A Brooks' type theorem 
for a signed coloring was proven 
by M\'{a}\v{c}ajov\'{a}, Raspaud, and \v{S}koviera \cite{MRS}.
Later, Fleiner and Wiener \cite{FW} gave a short proof,
using a DFS tree.

For a signed graph $(G,\sigma)$
and a mapping $L$ from $V(G)$ to $N_k$,
a \textit{signed $L$-coloring} is a signed coloring $f$ of $G$
such that $f(u) \in L(u)$ for each $u \in V(G)$.
Some results on signed $L$-coloring are showed in \cite{JKS, SS}.
In particular, 
Schweser and Stiebitz \cite{SS} gave 
a Brooks' type theorem for signed list-colorings.
In order to explain the exact statement,
we here introduce several definitions.

Let $(G,\sigma)$ be a signed graph.
A \textit{switching} at a vertex $v$
is defined as reversing the signs 
of all edges incident to $v$.
It is not difficult to see that 
a switching at any vertex does not change
the signed chromatic number of $(G,\sigma)$.
Note that 
a switching at $v$ in the sense 
of a signed $L$-coloring
corresponds to taking the mapping with $i \mapsto -i$
on $L(v)$.
Two signed graphs or 
two signs of a multigraph
are \textit{signed-equivalent}
or simply \textit{equivalent} 
if one is obtained from the other 
by a sequence of switchings.
A signed graph $(G,\sigma)$ is \textit{balanced}
if $\sigma$ is equivalent 
to the sign with all edges positive;
otherwise,
$(G,\sigma)$ is \textit{unbalanced}.
For a simple graph $G$,
a signed graph $(G^2, \sigma)$ is \textit{full}
if parallel edges with same end vertices have different signs on $\sigma$.

Then we are ready to state 
a Brooks' type theorem 
for a signed coloring.
This is an extension of Theorem \ref{BERT}.

\begin{thm}[Schweser and Stiebitz \cite{SS}]
\label{SSthm}
Let $(G,\sigma)$ be a signed graph,
where $G$ is connected,
and let $L$ be a mapping from $V(G)$ to $N_k$
with $|L(u)| \geq d_G(u)$ for each $u \in V(G)$.
Then 
$(G,\sigma)$ does not admit a signed $L$-coloring
if and only if each block of $(G,\sigma)$ is 
one of the following:
\begin{itemize}
\item A balanced $K_n$ for some integer $n$.
\item A balanced $C_n$ for some odd integer $n$.
\item An unbalanced $C_n$ for some even integer $n$.
\item A full $K_n^2$ for some integer $n$.
\item A full $C_n^2$ for some odd integer $n$.
\end{itemize}
\end{thm}


\subsection{Degree-DP-colorable graphs}

In the previous subsections,
we have seen two Brooks' type theorems,
namely that for DP-coloring and for signed coloring.
As we have seen,
a signed coloring can be regarded as a special case of a DP-coloring.
Roughly speaking,
Theorem \ref{BKPthm} is an ``almost'' improvement of Theorem \ref{SSthm}.
However,
when there is no desired coloring,
Theorem \ref{SSthm} completely determines 
which signs and list assignments forbid to have a desired coloring,
while Theorem \ref{BKPthm} only gives a structure of graphs.
Motivated by this situation,
we improve Theorem \ref{BKPthm} so that 
it completely covers Theorem \ref{SSthm}.

\medskip

Before explaining the exact statement of our main theorem,
we define three special graphs.
For two graphs $G$ and $H$ and a vertex $u$ of $G$,
\textit{blowing up $u$ to $H$}
is the operation 
to replace $u$ by $H$ so that
each vertex of $H$ is joined to every neighbor of $u$ in $G$.
Let $n$ and $t$ be positive integers.

\begin{itemize}
\item
The graph $H(n,t)$ 
is defined 
such that 
$\big\{(i,j,k) : i \in [n], \ j \in [n-1], \ k \in [t] \big\}$ 
is the vertex set and 
$(i,j,k)$ and $(i', j', k')$ are adjacent
if and only if
either
$i = i'$ or
$j = j'$.
See Figure \ref{H_nt_fig}.
Note that the graph $H(n,1)$ 
is isomorphic to 
$K_{n} \square K_{n-1}$,
and $H(n,t)$ is obtained from $H(n,1)$
by blowing up each vertex to a complete graph $K_t$.

\item
A graph $H$ is the \textit{$t$-fat ladder of length $n$} 
if
$\big\{(i,j,k) : i \in [n], \ j \in \{1,2\}, \ k \in [t] \big\}$ 
is the vertex set 
and 
$(i,j,k)$ and $(i', j', k')$ are adjacent
if and only if
either $i=i'$,
or $i' = i +1$ and $j = j'$,
where we define $n+1$ as $1$ for the subscript $i$.
In other words,
the $t$-fat ladder of length $n$
is obtained from the ladder of length $n$ (i.e.~$C_n \square K_2$)
by blowing up each vertex to a complete graph $K_t$.
The left of Figure \ref{example_fig} is a $1$-fat ladder
of length $4$.

\item
A graph $H$ is the \textit{$t$-fat M\"{o}bius ladder of length $n$} 
if
$\big\{(i,j,k) : i \in [n], \ j \in \{1, 2\}, \ k \in [ t ] \big\}$ 
is the vertex set 
and 
$(i,j,k)$ and $(i', j', k')$ are adjacent
if and only if
either 
\begin{itemize}
\item
$i=i'$, or 
\item
$i' = i +1$ and $j = j'$ for $1 \leq i \leq n-1$, or
\item
$i = n$, $i' = 1$ and $j \not= j'$.
\end{itemize}
In other words,
the $t$-fat M\"{o}bius ladder of length $n$
is obtained from the M\"{o}bius ladder of $2n$ vertices
by blowing up each vertex to a complete graph $K_t$.
The right of Figure \ref{example_fig} is a $1$-fat M\"{o}bius ladder
of length $4$.
\end{itemize}

\begin{figure}
\centering
\input{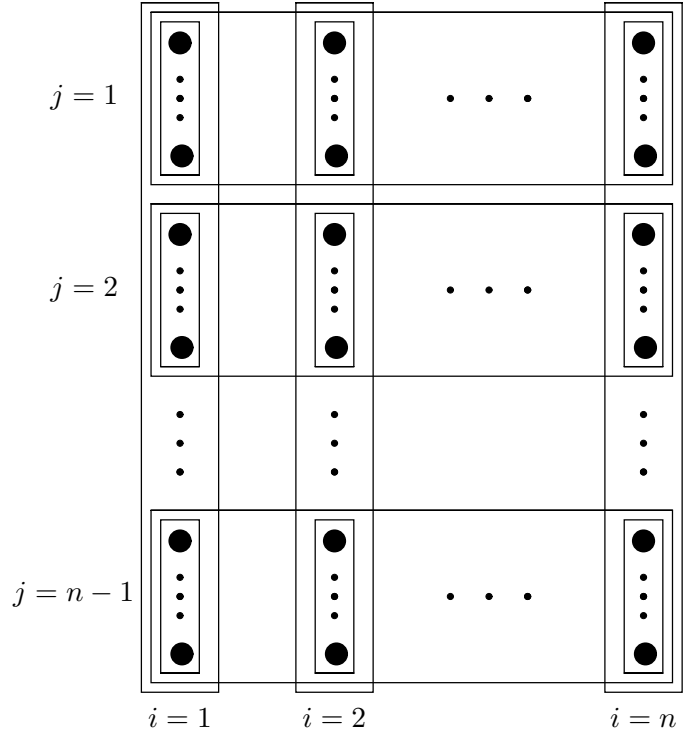}
\caption{The graph $H(n,t)$.
Each thin rectangle represents a clique.
In particular,
every minimal thin rectangle corresponds to $\big\{(i,j,k) : k \in [ t ] \big\}$ for some $i \in [n]$ and $j \in [n-1]$,
and contains exactly $t$ vertices.}
\label{H_nt_fig}
\end{figure}

For a vertex $u$ in a graph $G$ and a list assignment $L$ of $G$,
we denote $\{u\} \times L(u)$ by $\widetilde{L}(u)$
for simplicity.
Similarly, 
$\widetilde{L'}(u)$ denotes $\{u\} \times L'(u)$
for $L'(u) \subseteq L(u)$.

\medskip
Now we are ready to state our main theorem.

\begin{thm}
\label{mainthm}
Let $G$ be a connected multigraph, $L$ be a degree-list assignment of $G$,
and $\M_{L}$ be a matching assignment over $L$.
Then $G$ does not admit an $\M_{L}$-coloring
if and only if
each block of $G$ is isomorphic to
$K_n^t$ or $C_n^t$
for some integers $n$ and $t$
such that 
all of the following hold.
\begin{enumerate}[{\upshape (I)}]
\item
For each vertex $u$ in $G$,
the list assignment $L(u)$ has a partition 
$$\big\{L_B(u) : \text{$B$ is a block of $G$ containing $u$}\big\}$$ 
such that for any block $B$ containing $u$,
$$|L_B(u)| = 
\begin{cases}
t(n-1) & \text{if $B$ is isomorphic to $K_n^t$}, \\
2t & \text{if $B$ is isomorphic to $C_n^t$}.
\end{cases}
$$
\item
If $B$ is a block of $G$ isomorphic to $K_n^t$ 
for some integers $n$ and $t$,
then
$\bigcup_{u \in V(B)} \widetilde{L}_B(u)$
induces 
the graph $H(n,t)$ in the $\M_L$-cover of $G$,
where each set $\widetilde{L}_B(u)$ corresponds to
$\big\{(i_u,j,k) : j \in [n-1],\  k \in  [ t ]\big\}$
for some $i_u \in [n]$.
\item
If $B$ is a block of $G$ isomorphic to $C_n^t$ 
for some integers $n$ and $t$ with $n$ odd,
then 
$\bigcup_{u \in V(B)} \widetilde{L}_B(u)$
induces a $t$-fat ladder of length $n$ in the $\M_L$-cover of $G$,
where each set $\widetilde{L}_B(u)$ corresponds to
$\big\{(i_u,j,k) : j \in \{1,2\},\  k \in  [ t ]\big\}$
for some $i_u \in [n]$.
\item
If $B$ is a block of $G$ isomorphic to $C_n^t$ 
for some integers $n$ and $t$ with $n$ even,
then 
$\bigcup_{u \in V(B)} \widetilde{L}_B(u)$
induces a $t$-fat M\"{o}bius ladder of length $n$ in the $\M_L$-cover of $G$,
where each set $\widetilde{L}_B(u)$ corresponds to
$\big\{(i_u,j,k) : j \in \{1,2\},\  k \in  [ t ]\big\}$
for some $i_u \in [n]$.
\end{enumerate}

\end{thm}

We see that 
Theorem \ref{mainthm} is an extension of 
Theorems\ \ref{BERT}, \ref{BKPthm} and \ref{SSthm}.
Note that
$L_B$ in the ``only if'' part
is a degree-list assignment of $B$,
where $B$ is a block of $G$.
Furthermore,
we also see the following,
which will be used in our proof.
\begin{rem}
\label{degrem}
If each block of a graph $G$ is isomorphic to
$K_n^t$ or $C_n^t$
for some integers $n$ and $t$
and (I) holds for a list assignment $L$ of $G$,
then 
$|L(u)| = d_G(u)$ for each $u \in V(G)$.
\end{rem}

\section{Proof of Theorem \ref{mainthm}}

The proof of Theorem \ref{mainthm} uses a lemma,
which will be shown in the first subsection,
and then we give a proof of Theorem \ref{mainthm}.

\subsection{Greedy method to find an $\M_{L}$-coloring}
\label{greedysec}

As described in \cite{DP},
greedily choice of a color in $L(u)$
gives an $\M_{L}$-coloring 
for a $(k+1)$-list assignment $L$ of any $k$-degenerate graph $G$.
We use the same idea in this subsection
to obtain a useful lemma.

Let $G$ be a connected multigraph,
let $L$ be a list assignment of $G$,
and let $\M_{L}$ be a matching assignment over $L$.
For $u \in V(G)$ and $c \in L(u)$,
let $G^{(u)} := G - u$
and 
$$L^{(u,c)}(v) = L(v) - \big\{c' \in L(v): (u,c)(v,c') \in M_{L,uv} \big\}$$
for each $v \in V(G) - \{u\}$.
Note that the vertex $(u,c)$ has at most $\mu$ neighbors
in $\widetilde{L}(v)$
and $v$ lost $\mu$ edges from $G$ to $G^{(u)}$,
where $\mu$ is the multiplicity between $u$ and $v$ in $G$.
Thus,
if $L$ is a degree-list assignment of $G$,
then $L^{(u,c)}$ is a degree-list assignment of $G^{(u)}$.
We naturally denote by $\M_{L^{(u,c)}}$
the restriction of $\M_{L}$ into $G^{(u)}$ and $L^{(u,c)}$:
That is,
for each pair of vertices $v$ and $w$ in $G^{(u)}$,
$M_{L^{(u,c)},vw}$ is the union of matchings of $M_{L,vw}$
with end vertices contained in 
$\widetilde{L}^{(u,c)}(v)$
and $\widetilde{L}^{(u,c)}(w)$.
Let $\M_{L^{(u,c)}} = \big\{M_{L^{(u,c)},vw} : vw \in E\big(G^{(u)}\big) \big\}$.

Suppose that $G^{(u)}$ admits an $\M_{L^{(u,c)}}$-coloring,
that is,
there is an independent set $I_u$ in the $\M_{L^{(u,c)}}$-cover of $G$
with $|I_u| = |V(G^{(u)})| = |V(G)| -1$.
For each $(v,c_v)$ in $I_u$,
it follows from the choice of $L^{(u,c)}(v)$ that
$(v,c_v)$ is not a neighbor of $(u,c)$ in the $\M_L$-cover of $G$.
Therefore,
$I_u \cup \{(u,c)\}$ is an independent set 
in the $\M_L$-cover of $G^{(u)}$ with $|I| = |V(G)|$,
and hence $G$ admits an $\M_{L}$-coloring.
This gives the following lemma.

\begin{lemma}
\label{inductionlemma}
Let $G$ be a connected multigraph,
let $L$ be a list assignment of $G$,
and let $\M_{L}$ be a matching assignment over $L$.
For $u \in V(G)$ and $c \in L(u)$,
if $G^{(u)}$ admits an $\M_{L^{(u,c)}}$-coloring,
then $G$ admits an $\M_{L}$-coloring.
\end{lemma}

\subsection{Proof of Theorem \ref{mainthm}}

It is not difficult to check the ``if part''
(see \cite[Lemmas 7 and 8]{BKP}),
and hence we show the ``only if part'' by induction on $|V(G)|$.
Suppose that $G$ does not admit an $\M_{L}$-coloring.
By Theorem \ref{BKPthm},
each block of $G$ is isomorphic to
$K_n^t$ or $C_n^t$
for some integers $n$ and $t$.
Let $B_0$ be an end block of $G$
and let $H$ be the $\M_L$-cover of $G$.
\\

\Case{1.}
$B_0$ is isomorphic to $K_n^t$ 
for some integers $n$ and $t$.


Take a vertex $u$ in $B_0$ 
that is not a cut vertex of $G$.
For an element $c \in L(u)$,
consider the multigraph $G^{(u)} = G - u$ together with 
the list assignment $L^{(u,c)}$ of $G^{(u)}$
and the matching assignment $\M_{L^{(u,c)}}$.
Note that 
each block of $G^{(u)}$ is isomorphic to
$K_{n'}^{t'}$ or $C_{n'}^{t'}$
for some integers $n'$ and $t'$.
If 
some of (I)--(IV) do not hold
for $G^{(u)}$,
$L^{(u,c)}$
and 
$\M_{L^{(u,c)}}$,
then by induction hypothesis 
$G^{(u)}$ admits an $\M_{L^{(u,c)}}$-coloring.
However,
Lemma \ref{inductionlemma} gives
an $\M_L$-coloring in $G$,
a contradiction.
Therefore,
we may assume that 
all of (I)--(IV) hold:
That is,
(I) 
for each vertex $v$ in $G^{(u)}$,
the list assignment $L^{(u,c)}(v)$ has a partition 
$\big\{L'_B(v) : \text{$B$ is a block of $G^{(u)}$ containing $v$}\big\}$ 
such that for any block $B$ containing $v$,
$$|L'_B(v)| = 
\begin{cases}
t'(n'-1) & \text{if $B$ is isomorphic to $K_{n'}^{t'}$}, \\
2t' & \text{if $B$ is isomorphic to $C_{n'}^{t'}$},
\end{cases}
$$
and (II)--(IV) holds.
In particular,
Remark \ref{degrem}
implies that $|L^{(u,c)}(v)| = d_{G^{(u)}}(v)$
for each $v \in V(G^{(u)})$.
Let $B_0'$ be the subgraph of $G^{(u)}$ induced by $V(B_0) - \{u\}$.
Since $u$ is not a cut vertex of $G$,
$B_0'$ is a block of $G^{(u)}$ that is isomorphic to $K_{n-1}^t$.

By (II),
$\bigcup_{v \in V(B_0')} \widetilde{L}_{B_0'}'(v)$
induces 
the graph $H(n-1,t)$,
where each set $\widetilde{L}_{B_0'}'(v)$ corresponds to
$\big\{(i_v,j,k) : j \in [n-2],\  k \in  [ t ]\big\}$
for some $i_v \in [n-1]$.
For $v \in V(B_0')$ and $j \in [n-2]$,
let $L_j(v)$ be the set of elements $c_v \in L^{(u,c)}(v)$
such that $(v,c_v)$ corresponds to a vertex in
$\big\{(i_v,j,k) : k \in [ t ] \big\}$,
and let
$$L_{n-1}(v) = \big\{c_v \in L(v): (u,c)(v,c_v) \in M_{L,uv} \big\}.$$
Note that $|L_j(v)| = t$ for any $j \in [n-2]$.
So, if $v$ is not a cut vertex of $G$,
then 
$L(v) = \bigcup_{j=1}^{n-1} L_j(v)$ 
and
$|L_{n-1}(v)| = |L(v)| - t(n-2) \geq t$.
In particular,
we have $|L_{n-1}(v)| = t$.
Similarly, 
we obtain the same equality
even if $v$ is a cut vertex of $G$.
Let 
$$L_B(v) = 
\begin{cases}
L_B'(v) & \text{if $v \in V(G) - V(B_0)$ or $B \not= B_0$},\\
L(u) & \text{if $v = u$ and $B = B_0$},\\
L_{B_0'}'(v) \cup L_{n-1}(v)
 & \text{if $v \in V(B_0) - \{u\}$ and $B = B_0$}.
\end{cases}
$$
Note that this satisfies (I),
and also (II)--(IV) for all blocks $B$ with $B \not= B_0$.
Since $B_0$ is isomorphic to $K_n^t$,
it suffices to show (II) for $B_0$.

\medskip

Note that for any two vertices $v$ and $w$ in $B_0 - \{u\}$,
since $M_{L, vw}$ is the union of at most $t$ matchings
and $\widetilde{L}_j(v)$ and $\widetilde{L}_j(w)$ are 
all adjacent for $j \in [n-2]$,
there is no edge between $\widetilde{L}_j(v)$ and $\widetilde{L}_{j'}(w)$
if $j \not= j'$.

\medskip
Next, we show the following claim.

\begin{claim} \label{LjvLj'w-nbr} 
For $j \not= j'$ with $j,j' \in [n-1]$,
there is no element 
$c' \in L(u)$ such that $(u,c')$ has a neighbor both
in $\widetilde{L}_j(v)$ and in $\widetilde{L}_{j'}(w)$
for some $v,w \in V(B_0) - \{u\}$.
\end{claim}

\begin{proof}
Suppose that there exists an element $c' \in L(u)$
such that $(u,c')$ has a neighbor both
in $\widetilde{L}_j(v)$ and in $\widetilde{L}_{j'}(w)$
for some $v,w \in V(B_0) - \{u\}$ and $j,j' \in [n-1]$ with $j \not= j'$.
Note that in this case,
we have $n \geq 3$.
Then consider 
the list assignment $L^{(u,c')}$ of $G^{(u)}$
and 
the matching assignment $\M_{L^{(u,c')}}$.

\begin{itemize}
\item
Assume $v \not= w$.
By symmetry between $v$ and $w$,
we may assume that $w$ is not a cut vertex of $G$.
Since $(u,c')$ has a neighbor in $\widetilde{L}_{j'}(w)$,
there is an element $c_w \in L_j(w)$ that still remains in $L^{(u,c')}(w)$.
For $j \in [n-1]$,
since each vertex in $\widetilde{L}_j(w)$ does not have a neighbor
in $\widetilde{L}(v) - \widetilde{L}_j(v)$,
the existence of a neighbor of $(u,c')$ in $\widetilde{L}_j(v)$ implies
that 
the vertex $(w,c_w)$ has at most $t-1$ neighbors in $\widetilde{L}^{(u,c')}(v)$.
Thus,
$\widetilde{L}^{(u,c')}(w)$ cannot be 
the set 
$\big\{(i,j,k) : j \in [n-2],\  k \in [t] \big\}$
in $H(n-1,t)$.
Hence the $\M_{L^{(u,c')}}$-cover of $B_0'$ 
is not isomorphic to $H(n-1, t)$.
By induction hypothesis 
$G^{(u)}$ admits an $\M_{L^{(u,c')}}$-coloring,
but
Lemma \ref{inductionlemma} gives
an $\M_L$-coloring in $G$,
a contradiction.
\item
Assume $v = w$.
Then consider a vertex $z \not= v$ and an element $c_z$
in $L_{j}(z) \cup L_{j'}(z)$
that remains in $L^{(u,c')}(z)$.
(Such an element $c_z$ exists,
since $|L_{j}(z)| + |L_{j'}(z)| = 2t$ and 
$(u,c')$ has at most $t$ neighbors in $\widetilde{L}_{j}(z) \cup \widetilde{L}_{j'}(z)$.)
By symmetry,
we may assume that $c_z \in L_j(z)$.
Since 
$(u,c')$ has a neighbor in $\widetilde{L}_j(v)$ in $H$,
the vertex $(z,c_z)$ has at most $t-1$ neighbors in $\widetilde{L}^{(u,c')}(v)$.
Therefore,
by the same reason as above,
we see that the $\M_{L^{(u,c')}}$-cover of $B_0'$ 
is not isomorphic to $H(n-1, t)$,
and hence 
the induction hypothesis 
and Lemma \ref{inductionlemma} give
a contradiction, again.
\end{itemize}
Thus,
the claim holds.
\qquad 
$\blacksquare$
\end{proof}

\bigskip
Claim \ref{LjvLj'w-nbr} directly implies that
$L_{B_0}(u)$ can be divided into $n-1$ sets $L_1(u), \ldots , L_{n-1}(u)$
such that 
for each $j \in [n-1]$ and $c' \in L_j(u)$,
the vertex $(u,c')$ has neighbors
only in $\widetilde{L}_{B_0}(u) \cup 
\bigcup_{v \in V(B_0')} \widetilde{L}_j(v)$
in $H$.

\medskip
Next, we will prove the following Claim.

\begin{claim} \label{B_0-nbr}
(1) For each $j \in [n-1]$ and $c' \in L_j(u)$,
every vertex in $\bigcup_{v \in V(B_0')} \widetilde{L}_j(v)$ 
is a neighbor of $(u,c')$ in $H$.  

\medskip
(2) $\widetilde{L}_{n-1}(v)$ and $\widetilde{L}_{n-1}(w)$ are all adjacent
in $H$ for any $v,w \in V(B_0) - \{u\}$.
\end{claim}
\begin{proof}
(1) If for some $j \in [n-1]$ and $c' \in L_j(u)$,
some vertex in $\bigcup_{v \in V(B_0')} \widetilde{L}_j(v)$
is not a neighbor of the vertex $(u,c')$ in $H$,
then 
$|L^{(u,c')}(v)| \geq
|L(v)| - (t-1) \geq d_{G^{(u)}} (v) +1$,
and hence
the $\M_{L^{(u,c')}}$-cover of $B_0'$ 
is not isomorphic to $H(n-1, t)$.
(See Remark \ref{degrem}.)
By the induction hypothesis 
$G^{(u)}$ admits an $\M_{L^{(u,c')}}$-coloring,
but
Lemma \ref{inductionlemma} gives
an $\M_L$-coloring in $G$,
a contradiction.
Thus, (1) holds.
\qquad 
$\blacksquare$

(2) If some vertex in $\widetilde{L}_{n-1}(v)$ and 
some vertex in $\widetilde{L}_{n-1}(w)$ are not adjacent
in $H$ for $v,w \in V(B_0) - \{u\}$,
then
taking $c' \in L_{j}(u)$ with $j \in [n-2]$,
the vertex in $\widetilde{L}_{n-1}(v)$ 
has at most $t-1$ neighbors in $\widetilde{L}_{n-1}(w)$.
This implies again that
the $\M_{L^{(u,c')}}$-cover of $B_0'$ 
is not isomorphic to $H(n-1, t)$.
the induction hypothesis 
and Lemma \ref{inductionlemma} give a contradiction, again.
This completes the proof of Claim \ref{B_0-nbr}.
\qquad 
$\blacksquare$
\end{proof}

\bigskip
Claim \ref{B_0-nbr} implies that 
$\bigcup_{v \in V(B_0)} \widetilde{L}_{B_0}(v)$
induces 
the graph $H(n,t)$,
where each set $\widetilde{L}_{B_0}(v)$ corresponds to
$\big\{(i_v,j,k) : j \in [n-1],\  k \in  [ t ]\big\}$
for some $i_v \in [n]$.
This shows that $B_0$ also satisfies (II),
and completes the proof of Case 1.
\\

\Case{2.}
$B_0$ is isomorphic to $C_n^t$ for some integers $n$ and $t$.


Since $C_3^t$ is isomorphic to $K_3^t$,
we may assume that $n \geq 4$.
Let $u_{n-1}, u_n, u_{1}$ be the three consecutive vertices of $C_n^t$
such that $u_n$ is not a cut vertex of $G$.
Let $c_{n} \in L(u_{n})$.

\medskip
Suppose first that the vertex $(u_n,c_n)$ in $H$ has 
at most $t-1$ neighbors in $\widetilde{L}(u_1)$.
In this case,
consider the graph $G^{(u_n)} = G - u_n$,
the list assignment $L^{(u_n,c_n)}$ of $G^{(u_n)}$
and 
the matching assignment $\M_{L^{(u_n,c_n)}}$
as in Subsection \ref{greedysec}.
Since the vertex $(u_n,c_n)$ has 
at most $t-1$ neighbors in $\widetilde{L}(u_1)$,
we see that 
$$|L^{(u_n,c_n)}(u_1)| \geq |L(u_1)| - (t-1)
\geq d_G(u_1) - (t-1) = d_{G^{(u_n)}}(u_1) + 1.$$
Thus, by Remark \ref{degrem},
(I) does not hold 
for $G^{(u_n)}$,
$L^{(u_n,c_n)}$ and $\M_{L^{(u_n,c_n)}}$,
and hence 
by induction hypothesis 
$G^{(u_n)}$ admits an $\M_{L^{(u_n,c_n)}}$-coloring.
However,
Lemma \ref{inductionlemma} gives
an $\M_L$-coloring in $G$,
a contradiction.
Therefore, the vertex $(u_n,c_n)$ in $H$ has 
exactly $t$ neighbors in $\widetilde{L}(u_1)$.
By the same argument,
the vertex $(u_n,c_n)$ in $H$ has 
exactly $t$ neighbors in $\widetilde{L}(u_{n-1})$.
Similarly,
we can prove that $|L(u_1)| = |L(u_{n-1})| = 2t$.

\medskip

For $i= 1, n-1$,
let $L_1(u_i)$ be the set of elements $c_i \in L(u_i)$
such that 
$(u_i,c_i)$ is a neighbor of $(u_{n},c_{n})$ in $H$,
and let $L_2(u_i) = L(u_i) - L_1(u_i)$.
Since $|L_1(u_i)| = t$ and $|L(u_i)| = 2t$,
we have $|L_2(u_i)| = |L(u_i)| - |L_1(u_i)| = t$.
Then we construct the graph $G'$
from $G - u_{n}$ by adding $t$ multiple edges connecting $u_{n-1}$ and $u_1$.
Let $B_0'$ be the subgraph of $G'$ induced by $V(B_0) - \{u_n\}$.
Since $n \geq 4$
and $u_n$ is not a cut vertex of $G$,
$B_0'$ is a block of $G'$ that is isomorphic to $C_{n-1}^t$.
Let $L'$ be the restriction of $L$ into $V(G')$,
let $M'_{L',vw} = M_{L,vw}$ for $vw \in E(G') - \{u_{n-1}u_1\}$,
let $M'_{L',u_{n-1}u_1}$ be the set of 
all possible edges between 
$\widetilde{L}_j(u_{n-1})$ and $\widetilde{L}_{j'}(u_{1})$
for $\{j,j'\} = \{1,2\}$,
and let $\M'_{L'} = \big\{M'_{L',vw} : vw \in E(G') \big\}$.

\medskip
Suppose that 
$G'$ admits an $\M'_{L'}$-coloring,
that is,
the $\M'_{L'}$-cover of $G'$ contains an independent set $I'$
of size $|V(G')| = |V(G)| -1$.
For $i = 1,n-1$, let $c_i \in L(u_i)$ with $(u_i,c_i) \in I'$.
If $c_{1} \in L_2(u_{1})$ and $c_{n-1} \in L_2(u_{n-1})$,
then $I' \cup \{(u_n,c_n)\}$ is an independent set in $H$ of size $|V(G)|$,
a contradiction.
Thus,
we may assume that 
either $c_1 \in L_1(u_1)$ or $c_{n-1} \in L_1(u_{n-1})$.
Since $(u_1,c_1)$ and $(u_{n-1},c_{n-1})$ are not adjacent
in the $\M'_{L'}$-cover of $G'$,
the choice of $M'_{L',u_{n-1}u_1}$ implies that
both $c_1 \in L_1(u_1)$ and $c_{n-1} \in L_1(u_{n-1})$ hold.
Furthermore,
for any $c_n' \in L(u_n)$,
since $I' \cup \{(u_n,c_n')\}$ is not an independent set in $H$,
the vertex $(u_n,c_n')$ must be a neighbor of 
either $(u_1,c_1)$ or $(u_{n-1},c_{n-1})$.
Since $(u_n,c_n)$ is a neighbor of both 
$(u_1,c_1)$ and $(u_{n-1},c_{n-1})$,
there are at least $|L(u_1)| + 1 \geq 2t+1$ edges in $H$
between $\big\{(u_1,c_1), (u_{n-1},c_{n-1}) \big\}$
and $\widetilde{L}(u_n)$.
This contradicts that 
both $(u_1,c_1)$ and $(u_{n-1},c_{n-1})$ have 
at most $t$ neighbors in $\widetilde{L}(u_n)$.
Therefore,
we have 
\\

\begin{tabular}{ll}
$(P1):$ &
the $\M'_{L'}$-cover of $G'$ contains no independent set 
of size $|V(G')|$.
\end{tabular}
\\

By induction hypothesis,
all of (I)--(IV) hold for $G'$, $L'$ and $\M'_{L'}$:
That is,
(I) 
for each vertex $v$ in $G'$,
the list assignment $L'(v)$ has a partition 
$\big\{L'_B(v) : \text{$B$ is a block of $G'$ containing $v$}\big\}$ 
such that for any block $B$ containing $v$,
$$|L'_B(v)| = 
\begin{cases}
t'(n'-1) & \text{if $B$ is isomorphic to $K_{n'}^{t'}$}, \\
2t' & \text{if $B$ is isomorphic to $C_{n'}^{t'}$},
\end{cases}
$$
and (II)--(IV) holds.
For $v \in V(G)$ and for a block $B$ containing $v$,
let 
$$L_B(v) = 
\begin{cases}
L_B'(v) & \text{if $v \not= u_n$},\\
L(v) & \text{if $v = u_n$ and $B = B_0$}.\\
\end{cases}
$$
Note that this satisfies (I),
and also (II)--(IV) for all blocks $B$ with $B \not= B_0$.
Since $B_0$ is isomorphic to $C_n^t$,
it suffices to show (III) and (IV) for $B_0$.

\medskip
Suppose that there are three elements 
$c_{n-1}' \in L_1(u_{n-1})$,
$c_n' \in L(u_n)$ and $c_1' \in L_2(u_1)$
such that neither $(u_{n-1},c_{n-1}')$ nor $(u_1,c_1')$ is
a neighbor of $(u_n,c_n')$ in $H$.
Let $M^-_{L',vw} = M_{L,vw}$ for $vw \in E(G') - \{u_{n-1}u_1\}$,
let 
$$M^-_{L',u_{n-1}u_1} = M'_{L',u_{n-1}u_1} - \big\{(u_{n-1},c_{n-1}')(u_1,c_1')\big\},$$
and let 
$\M^-_{L'} = \big\{M^-_{L',vw} : vw \in E(G') \big\}$.
Since $G'$, $L'$ and $\M'_{L'}$ satisfy (I)--(IV)
and the $\M^-_{L'}$-cover of $G'$ is obtained
from the $\M'_{L'}$-cover of $G'$ by deleting one edge,
we see that 
$G'$, $L'$ and $\M^-_{L'}$ do not satisfy (III) or (IV).
Therefore,
by the induction hypothesis,
$G'$ admits an $\M^-_{L'}$-coloring:
That is,
the $\M^-_{L'}$-cover of $G'$ contains an independent set $I^-$
of size $|V(G')| = |V(G)| -1$.
If either $(u_{n-1},c_{n-1}') \not\in I^-$ or $(u_1,c_1') \not\in I^-$,
then 
$I^-$ is also an independent set of the $\M'_{L'}$-cover of $G'$,
which contradicts $(P1)$.
Thus,
$(u_{n-1},c_{n-1}') \in I^-$ and $(u_1,c_1') \in I^-$
and hence 
$I^- \cup \big\{(u_n, c_n')\big\}$ is an independent set in $H$
of size $|V(G)|$, again a contradiction.
Therefore, the following holds.
\\

\begin{tabular}{ll}
$(P2):$ & 
There are no three elements 
$c_{n-1}' \in L_1(u_{n-1})$,
$c_n' \in L(u_n)$ and $c_1' \in L_2(u_1)$  \\
&
such that neither $(u_{n-1},c_{n-1}')$ nor $(u_1,c_1')$ is
a neighbor of $(u_n,c_n')$ in $H$.
\end{tabular}
\\

\noindent
Furthermore,
by symmetry,
there are no three elements 
$c_{n-1}' \in L_2(u_{n-1})$,
$c_n' \in L(u_n)$ and $c_1' \in L_1(u_1)$
such that neither $(u_{n-1},c_{n-1}')$ nor $(u_1,c_1')$ is
a neighbor of $(u_n,c_n')$ in $H$.

\medskip
Then, we can show the following property.

\medskip
\begin{tabular}{ll}
$(P3):$ & For any $c_n' \in L(u_n)$,
either all vertices in $\widetilde{L}_1(u_{n-1})$
or 
all vertices in $\widetilde{L}_2(u_{n-1})$ 
\\ &
are neighbors of $(u_n, c_n')$.
\end{tabular}
\\

Suppose that for $c_n' \in L(u_n)$,
there are a vertex in $\widetilde{L}_1(u_{n-1})$
and a vertex in $\widetilde{L}_2(u_{n-1})$
neither of which are neighbors of $(u_n,c_n')$ in $H$.
Since $(u_n,c_n')$ has at most $t$ neighbors in $\widetilde{L}(u_1)$,
there exists an element $c_1' \in L(u_1)$ 
such that $(u_1,c_1')$ is not a neighbor of $(u_n,c_n')$.
We here assume $c_1' \in L_2(u_1)$,
but the other case is symmetric.
(Here we do not use $c_n$).
Let $c_{n-1}' \in L_1(u_{n-1})$ 
such that $(u_{n-1}, c_{n-1}')$ is not a neighbor of $(u_n,c_n')$.
Then 
the three elements $c_{n-1}' \in L_1(u_{n-1})$,
$c_n' \in L(u_n)$ and $c_1' \in L_2(u_1)$
contradict $(P2)$.  Thus $(P3)$ holds.

\medskip

Thus we have followings:

\begin{itemize}

\item
For any $c_n' \in L(u_n)$,
since $(u_n,c_n')$ has at most $t$ neighbors in $\widetilde{L}(u_{n-1})$
and $|\widetilde{L}_1(u_{n})| = |\widetilde{L}_2(u_{n})| = t$,
(P3) directly implies that
the vertex $(u_n,c_n')$ has no neighbor
either in $\widetilde{L}_1(u_{n-1})$ or in $\widetilde{L}_2(u_{n-1})$.

\item
Thus,
$L(u_n)$ can be divided into two sets $L_1(u_n)$ and $L_{2}(u_n)$
such that 
for each $j \in \{1,2\}$ and $c_n' \in L_j(u_n)$,
the set of neighbors of $(u_{n},c_{n}')$ in $\widetilde{L}(u_{n-1})$
is $\widetilde{L}_j(u_{n-1})$.
Note that $c_n \in L_1(u_n)$.

\item
Let $c_n' \in L_2(u_n)$.
Since $(u_n, c_n')$ has no neighbors in $\widetilde{L}_1(u_{n-1})$,
$(P2)$ implies that 
all vertices in $\widetilde{L}_2(u_{1})$ are neighbors of $(u_n,c_n')$ in $H$.
Since $(u_n,c_n')$ has at most $t$ neighbors in $\widetilde{L}(u_{1})$,
no vertices in $\widetilde{L}_1(u_{1})$ are neighbors of $(u_n,c_n')$ in $H$.

\item
Similarly,
for $c_n' \in L_1(u_n)$,
all vertices in $\widetilde{L}_1(u_{1})$ are neighbors of $(u_n,c_n')$ in $H$
and 
no vertices in $\widetilde{L}_2(u_{1})$ are neighbors of $(u_n,c_n')$.

\end{itemize}

Recall that
$M'_{L',u_{n-1}u_1}$ is the set of 
all possible edges between 
$\widetilde{L}_j(u_{n-1})$ and $\widetilde{L}_{j'}(u_{1})$
for $\{j,j'\} = \{1,2\}$.
Then
the above conditions imply that
$M_{L,u_{n-1}u_n}$ is the set of 
all possible edges between 
$\widetilde{L}_j(u_{n-1})$ and $\widetilde{L}_{j}(u_{n})$
for $j \in \{1,2\}$,
and $M_{L,u_{n}u_1}$ is the set of 
all possible edges between 
$\widetilde{L}_j(u_{n})$ and $\widetilde{L}_{j}(u_{1})$
for $j \in \{1,2\}$.

\medskip
When $n$ is odd,
it follows from (IV) for $B_0'$ that
$\bigcup_{v \in V(B_0')} \widetilde{L'}_{B_0'}(v)$
induces 
a $t$-fat M\"{o}bius ladder of length $n-1$
in the $\M'_{L'}$-cover of $G'$.
Then it is not difficult to see that
$\bigcup_{v \in V(B_0)} \widetilde{L}_{B_0}(v)$
induces 
a $t$-fat ladder of length $n$.
Therefore,
(III) and trivially (IV) hold for $B_0$.

\medskip
By the same way,
we show that when $n$ is even,
$\bigcup_{v \in V(B_0)} \widetilde{L}_{B_0}(v)$
induces a $t$-fat M\"{o}bius ladder of length $n$.
Therefore,
(IV) and trivially (III) hold for $B_0$.
This completes the proof of Theorem \ref{mainthm}.
\qed


\end{document}